\documentclass[graybox]{svmult}

\usepackage{mathptmx}       % selects Times Roman as basic font
\usepackage{helvet}         % selects Helvetica as sans-serif font
\usepackage{courier}        % selects Courier as typewriter font

\usepackage{amssymb,amsmath,enumerate,epsfig}%,epsf}%t1enc
\usepackage{amsfonts,color}
\usepackage{a4,latexsym}
\usepackage[all]{xy}
\usepackage{multirow}
\usepackage{makeidx}         % allows index generation
\usepackage{graphicx}        % standard LaTeX graphics tool
\usepackage{multicol}        % used for the two-column index

\usepackage{hyperref}

\makeindex             % used for the subject index

\hypersetup{pdftitle=2lect}

\def\podu{{\sf pd}}   % Poincare dual
\def\per{{\sf pm}}      % period map/matrix
\def\perr{{\sf q}}        %period matrix.....
\def\perdo{{\mathcal K}}   %period domain
\def\sfl{{\mathrm F}} %Space of filtrations
\def\Z{\mathbb{Z}}                   %Integer  numbers
\def\Q{\mathbb{Q}}                   %Rational  numbers
\def\C{\mathbb{C}}                   %Complex numbers
\def\O{{\mathcal O}}                     %ring of integers of a number field
\def\R{\mathbb{R}}                   %real numbers
\def\N{\mathbb{N}}                   %natural numbers
\def\uhp{{\mathbb H}}                %upper half-plane
\def\Mat{{\rm Mat}}              %Matrices
\newcommand{\mat}[4]{
     \begin{pmatrix}
            #1 & #2 \\
            #3 & #4
       \end{pmatrix}
    }                                %two by two matrices
\newcommand{\matt}[2]{
     \begin{pmatrix}                 % one by two matrix
            #1   \\
            #2
       \end{pmatrix}
    }
\def\dR{{\rm dR}}                %The subindex dR standing for de Rham
\def\Hb{{\mathcal H}}                    %Hodge bundle
\def\GL{{\rm GL}}                %The liner group
\def\F{{\mathcal F}}                     %Hodge filtration bundle

\def\Im{{\rm Im}}                %imaginary
\newcommand\SL[2]{{\rm SL}(#1, #2)}    %SL(2,Z)
\def\L{{\mathcal L}}                     %The moduli of polarized lattices in a
\def\Aut{{\rm Aut}}              %Automorphism group of a vectorspace
\def\any{R}                          %Any subring of the field of complex
\newcommand\ovl[1]{\overline{#1}}    %Conjugation of #1.

\def\tr{{\mathsf t}}                 %Transposition of matrices
\def\perdo{{ X}}   %period domain
\def\pedo{  P}                  %Period domain

\begin{document}

\title*{Quasi-modular forms attached to Hodge structures }

\author{Hossein Movasati}
\institute{Hossein Movasati \at  Instituto de Matem\'atica Pura e Aplicada, IMPA,
Estrada Dona Castorina, 110,
22460-320, Rio de Janeiro, RJ, Brazil, 
\email{hossein@impa.br} }
%\urladdr{http://www.iag.uni-hannover.de/$\sim$schuett/}

\maketitle

 \abstract*{
The space $D$  of Hodge structures on a fixed polarized lattice is known as a Griffiths period domain and its quotient by the isometry group of the lattice is the moduli of polarized Hodge structures 
of a fixed type. When $D$ is a Hermition symmetric domain then we have automorphic forms on $D$, which according to the Baily-Borel theorem, give an algebraic structure to 
the above-mentioned  moduli space.  In this article we slightly modify this picture by considering the space $U$ of polarized lattices in a fixed complex vector space with a fixed Hodge filtration 
and polarization. It turns out that the isometry group of the filtration and polarization, which is an algebraic group, acts on $U$ and the quotient is again the moduli of 
polarized Hodge structures. This formulation leads us to a notion of quasi-automorphic  forms which generalizes quasi-modular forms attached to elliptic curves.     
}
 \abstract{
The space $D$  of Hodge structures on a fixed polarized lattice is known as Griffiths period domain and its quotient by the isometry group of the lattice is the moduli of polarized Hodge structures 
of a fixed type. When $D$ is a Hermition symmetric domain then we have automorphic forms on $D$, which according to Baily-Borel theorem, they give an algebraic structure to 
the mentioned  moduli space.  In this article we slightly modify this picture by considering the space $U$ of polarized lattices in a fixed complex vector space with a fixed Hodge filtration 
and polarization. It turns out that the isometry group of the filtration and polarization, which is an algebraic group, acts on $U$ and the quotient is again the moduli of 
polarized Hodge structures. This formulation leads us to a notion of quasi-automorphic  forms which generalizes quasi-modular forms attached to elliptic curves.     
}
\medskip

{\noindent
\textbf{MSC (2010):} %\subjclass[2010]{
32G20, 11F46, 14C30, 14J15}

%\subjclass[2010]{14J28; 11F03, 11G05, 11G15, 11G25, 11G35, 14G05, 14G15, 14G25, 14J10, 14J27}
%
\keywords{Polarized Hodge structure, period map, algebraic de Rham cohomology}
%
%
%\thanks{Partial funding from DFG grant Hu 337/6-1 is gratefully acknowledged}
%
%\dedicatory{Dedicated to Tetsuji Shioda on the occasion of his 70th birthday}
%

\date{March 4, 2012}

In 1970 Griffiths in his article \cite{gr70} introduced the 
period domain $D$  and described a project to enlarge $D$ to a moduli space of degenerating polarized
Hodge structures. He also asked  for the existence of a certain automorphic form theory for $D$, generalizing the usual notion of automorphic forms on Hermitian symmetric domains. 
Since then there have been much effort made on  the first
part of Griffiths's project (see \cite{kaus00, ho10} and the 
references there). For the second part Griffiths himself introduced the theory of  automorphic cohomology; however,  the generating function role of automorphic forms is somewhat 
lacking in this theory.

Some years ago, I was looking for some analytic spaces over $D$ for which one may state the Baily-Borel theorem on the unique algebraic structure of quotients
of Hermitian symmetric domains by discrete arithmetic groups.  I realized that even
in the simplest case of Hodge structures, namely  $h^{01}=h^{10}=1$,  
such spaces are not well studied. This led me to the definition of a class of
holomorphic functions on the Poincar\'e upper-half-plane which generalize the classical modular
forms (see \cite{ho06-2}). Since a differential operator acts on them I called them differential modular forms. Soon after I realized that 
such functions play a central role in mathematical physics and, in particular, in mirror symmetry (see \cite{kon95} and the references therein). Inspired by this  special 
case of Hodge structures with its fruitful applications, I felt the necessity to develop as much as possible similar theories for an arbitrary type of Hodge structure.   

In this note we construct an analytic variety $U$ and an action of an 
algebraic group $G_0$ on $U$ from the right such that $U/ G_0$ is the 
moduli space of polarized Hodge structures of a fixed type. We may pose the following algebraization problem for $U$, in parallel to
the Baily-Borel theorem in \cite{babo66}:  construct  functions on $U$ which have some automorphic properties with respect 
to the action of $G_0$ and have some finite growth when a Hodge structure degenerates. There  must be enough  of them in order to enhance $U$ with a canonical structure of an 
algebraic variety such that the action 
of $G_0$ is algebraic. 
% However, we still far from this the realization of this dream. 
In the case for which the Griffiths period domain is Hermitian symmetric, for instance for the Siegel upper half-plane, this problem seems to be promising but needs a 
reasonable amount of work if one wants to construct such functions through the inverse of the generalized period maps (see \S\ref{h10=h01}). Among them are calculating explicit affine coordinates
in certain moduli spaces and calculating Gauss-Manin connections. Some main ingredients of such a study for K3 surfaces endowed with polarizations is 
already done by many authors, see for instance \cite{do10-1} and the references therein. 
For the case in which the Griffiths period domain is not Hermitian symmetric, we reformulate the algebraization problem further (see \S \ref{13sep2010})  and we solve it for the  Hodge numbers 
$h^{30}=h^{21}=h^{12}=h^{03}=1$ (see \S\ref{1111} and \cite{ho10*}). This gives us a first example of quasi-automorphic forms theory attached to a period domain 
which is not Hermitian symmetric.

The realization of the algebraization problem in the case of elliptic curves and the corresponding Hodge numbers $h^{10}=h^{01}=1$ clarifies many details 
of the previous paragraph; therefore, I explain it here (for more details see \cite{ho06-2}).  In this case
$U=\SL 2\Z\backslash \pedo$, where
$$
\pedo:=\{\mat {x_1}{x_2}{x_3}{x_4}\in \SL 2\C\mid \Im(x_1\ovl{x_3})>0\}.
$$ 
In order to find an algebraic structure on $U$ we 
work with the following family of elliptic curves:
$$
E_t: \ y^2-4(x-t_1)^3+t_2(x-t_1)+t_3=0,\ %\  27t_3^2-t_2^3\not=0.
$$
where  the parameter $t=(t_1,t_2,t_3)$ is a point of the  affine variety
$$
T:=\{(t_1,t_2,t_3)\in \C^3 \mid  27t_3^2-t_2^3\neq0\}.
$$
The generalized period map
\begin{equation}
\label{3sep2010}
 \per: T\to U,
\end{equation}
%The biholomorphism (\ref{3sep2010}) is given by the generalized period map 
$$
t\mapsto
\left [\frac{1}{\sqrt{-2\pi i}}\mat
{\int_{\delta_1}\frac{dx}{y}}
{\int_{\delta_1}\frac{xdx}{y}}
{\int_{\delta_2} \frac{dx}{y}}
{\int_{\delta_2} \frac{xdx}{y}} \right]
$$
is in fact a biholomorphism. 
Here, $[\cdot]$ means the equivalence class and $\{\delta_1,\delta_2\}$ is a basis of the $\Z$-module 
$H_1(E_t,\Z)$ with $\langle \delta_1,\delta_2\rangle=-1$.  
The algebraic group
$$
G_0= \{\mat{k}{k'}{0}{k^{-1}} \mid  \ k,k'\in\C,\  k\neq 0  \}
$$ 
acts from the right on $U$ by the usual multiplication of matrices.
Under $\per$ the action of $G_0$ is given by
$$
t\bullet g=(t_1k^{-2}+k'k^{-1},
t_2k^{-4}, t_3k^{-6}) ,
$$ 
$$
 t=(t_1,t_2,t_3)\in\C^3,\ 
 g=\mat {k}{k'}{0}{k^{-1}}\in G_0.
$$
%such that the intersection matrix in this basis is $\mat{0}{1}{-1}{0}$. 
In fact, $T$ is the moduli space of pairs $(E,\{\omega_1,\omega_2\})$, where $E$ is an elliptic curve and $\{\omega_1,\omega_2\}$ is a basis 
of $H_\dR^1(E)$ such that $\omega_1$ is represented by a differential form of the first kind and $\frac{1}{2\pi i}\int_{E} \omega_1\cup \omega_2=1$.

The algebra of quasi-modular forms arises in the following way: We consider the composition  of maps 
\begin{equation}
\label{khoshgel?}
\uhp \stackrel{i}{\hookrightarrow} P\to U\stackrel{\per^{-1}}{\to} T\hookrightarrow\tilde T,
\end{equation}
where $\uhp=\{\tau\in\C\mid \Im(\tau)>0\}$ is the upper half-plane,  
$$
i:\uhp\to P,\  i(\tau)=\mat{\tau}{-1}{1}{0},
$$ 
$P\to U$ is the quotient map and $\tilde T=\C^3$ is the underlying complex 
manifold of the affine variety ${\rm Spec}(\C[t_1,t_2,t_3])$. The pullback of the function ring $\C[t_1,t_2,t_3]$ of $\tilde T$ by the composition $\uhp\to \tilde T$
 is a $\C$-algebra which we call 
the $\C$-algebra of quasi-modular forms for $\SL 2\Z$. Three Eisenstein series
\begin{equation}
\label{eisenstein}
g_i(\tau)=a_k{\Big (}1+b_k\sum_{d=1}^\infty d^{2k-1}\frac{e^{2\pi i d \tau}}{1-e^{2\pi i d \tau}}{\Big )},\ \  k=1,2,3, 
\end{equation}
where 
$$
(b_1,b_2,b_3)=
(-24, 240, -504),\ \ (a_1,a_2,a_3)= (\frac{2\pi i}{12},12(\frac{2\pi i}{12})^2 ,8(\frac{2\pi i}{12})^3)
$$
are obtained by taking the pullback of the $t_i$'s. Our reformulation of the algebraization problem is based on (\ref{khoshgel?}) and the pullback argument, see \S\ref{13sep2010}. 

We fix some notations from linear algebra. For a basis $\omega_1,\omega_2,\ldots,\omega_h$ of a vector space we denote by $\omega$ an $h\times 1$ matrix whose entries are  the 
$\omega_i$'s. In this way we also say that $\omega$ is a basis of the vector space. 
If there is no danger of confusion we also use $\omega$ to denote an element of the vector space. 
We use $A^\tr$ to denote the transpose
of the matrix $A$.  Recall that if $\delta$ and $\omega$ are two bases  of a vector space, 
$\delta=p\omega$ for some $p\in \GL(h,\C)$ and a bilinear form on $V_0$ in the 
basis $\delta$ (resp. $\omega$) has
the matrix form $A$ (resp. $B$) then $pBp^\tr=A$. By $[a_{ij}]_{h\times h}$ we mean an $h\times h$ matrix whose $(i,j)$ entry is $a_{ij}$. 
%-----------------------------------------------------------------------------
\section{Moduli of polarized Hodge structures}
In this section we define the generalized period domain $U$ and we explain its comparison with the classical Griffiths period domain. 
\subsection{The space of polarized lattices}
\label{lattice}
We fix a $\C$-vector space $V_0$ \index{ $V_0$, a $\C$-vector space} of dimension $h$, a natural number $m\in\N$ 
and a $h\times h$ integer-valued matrix $\Psi_0$\index{$\Psi_0$, intersection matrix} such that the
associated bilinear form 
$$
\Z^h\times \Z^h\rightarrow \Z,\ (a,b)\rightarrow 
a^\tr\Psi_0 b
$$ 
is non-degenerate, symmetric if $m$ is even and skew if $m$ is 
odd. Note that, in the case of $\Z$-modules, by non-degenerate we mean that
the associated morphism 
$$
\Z^h\rightarrow (\Z^h)^\vee,\ a\rightarrow(b\rightarrow a^\tr\Psi_0 b)
$$ 
is an isomorphism, 
where $\vee$ means the dual of a $\Z$-module.

A lattice  $V_\Z$ \index{lattice} in $V_0$ is a $\Z$-module generated 
by a basis of
$V_0$. A polarized lattice \index{polarized lattice} $(V_\Z,\psi_\Z)$ of type
$\Psi_0$ is a lattice $V_\Z$ together with a bilinear map
$\psi_\Z:V_\Z\times V_\Z\rightarrow \Z$ such that in a $\Z$-basis of $V_\Z$,  
$\psi_\Z$ has the form $\Psi_0$. \index{$\psi_\Z$, polarization}

Let $\L$ be the set of polarized lattices of type $\Psi_0$ in $V_0$. 
It has a canonical structure of a complex manifold of dimension $\dim_\C(V_0)^2$. One can take a local chart around 
$(V_\Z,\psi_\Z)$ by fixing a basis of the $\Z$-module $V_\Z$.
\index{$\L$, space of
polarized lattices}
Usually, we denote an element of $\L$ by $x,y,\ldots$ and the associated 
lattice (resp. bilinear form) by $V_\Z(x),V_\Z(y),\ldots $ 
(resp. $\psi_\Z(x),\psi_\Z(y),\ldots$). Let $R$ be any subring of $\C$. For
instance, $R$ can be $\Q$, $\R$, $\C$, $\Z$. We define
$$
V_\any(x):=V_\Z(x)\otimes_\Z \any \hbox{ and } 
\psi_\any(x):V_\any(x)\times V_\any(x)
\rightarrow \any \hbox{ the induced map. }
$$
\index{$V_\any(x)$, $\any$-module}\index{$\psi_\any(x)$, polarization}
%Let  $I:=\{1,2,\ldots,h\}$. 
Conjugation with respect to $x\in\L$
of an element $\omega=\sum_{i=1}^h a_i\delta_i\in V_0$, where $V_\Z(x)=\sum_{i=1}^h \Z \delta_i$, 
is defined by 
$$
\overline{\omega}^x:=\sum_{i=1}^h \overline a_i\delta_i,
$$\index{$\overline{a}^x$, conjugation of $a$ with respect to $x$}
where $\overline s,\ s\in \C$ is the usual conjugation of complex numbers. 
%---------------------------------------------
%---------------------------------------------------------------------------
\subsection{Hodge filtration}
\label{hodgefilsection}
We fix Hodge numbers \index{Hodge numbers}
$$h^{i,m-i}\in \N\cup\{0\},\ 
h^i:=\sum_{j=i}^m h^{j,m-j},\  i=0,1,\ldots,m,\ 
h^0=h
$$\index{$h^{i,m-i}$, $h^i$, Hodge numbers}
a filtration 
\begin{equation}
\label{hodgefiltration}
F^\bullet_0: \{0\}=F^{m+1}_0\subset F^m_0\subset \cdots\subset F^1_0\subset 
F^0_0=V_0,\  dim(F_0^i)=h^i
\end{equation}
on $V_0$ and a bilinear form
$$
\psi_0: V_0\times V_0\to \C
$$
such that in a basis of $V_0$ its matrix is $\Psi_0$ and it satisfies
\begin{equation}
\label{17oct2011}
\psi_0(F^i_0,F^j_0)=0,\ \forall i,j,\  i+j>m.
\end{equation}
A basis $\omega_i,i=1,2,\ldots,h$ of $V_0$ is compatible with 
the filtration $F^\bullet_0$ if $\omega_i,\ i=1,2,\ldots,h^i$ is a 
basis of $F^i_0$ for all $i$. 
It is sometimes convenient to fix a basis $\omega_i,\ i=1,2,\ldots,h$ of $V_0$ which is compatible with the filtration $F^\bullet_0$ and such
that the polarization matrix $[\psi_0( \omega_i,\omega_j)]$ is a fixed matrix $\Phi_0$:
$$
[\psi_0( \omega_i,\omega_j)]=\Phi_0.
$$ The matrices $\Psi_0$ 
and $\Phi_0$ are not necessarily the same.
For any $x\in \L$ we define
$$
H^{i,m-i}(x):=F_0^i\cap \overline{F_0^{m-i}}^x
$$\index{$F_0^\bullet$, Hodge filtration}
and the following properties for $x\in \L$:
\begin{enumerate}
\item 
$\psi_\C(x)=\psi_0$;
\item
$V_0=\oplus_{i=0}^m H^{i,m-i}(x)$;
\item
$(-1)^{\frac{m(m-1)}{2}+i}(\sqrt{-1})^{-m}\psi_\C(x)(\omega,\overline \omega^x)>0,\ 
\forall \omega\in H^{i,m-i}(x),\ 
\omega\ne 0$.
\end{enumerate}
Throughout the text we call these properties  P1, P2 and P3.
Fix a polarized lattice $x\in \L$.
P1 implies that 
$$
\psi_0(H^{i,m-i}(x),H^{j,m-j}(x))=0 \hbox{ except for } i+j=m.
$$
%The proof is as follows: we have $\psi_\C(H^{i,m-i}(x),H^{j,m-j}(x))=
%\psi_\C(F_0^i\cap \overline{F_0^{m-i}}^x,F_0^j\cap \overline{F_0^{m-j}}^x)=0$. 
This is because  if $i+j>m$ then $\psi_0(F^i_0,F^j_0)=0$ and if
$i+j<m$  then $\psi_0(\overline{F^i_0}^x,\overline{F^j_0}^x)=0$.
We have also $\sum_i H^{i,m-i}(x)  =\oplus_{i} H^{i,m-i}(x)$ if and only if
\begin{equation}
\label{1oct2005}
F^i_0\cap \overline{F^j_0}^x=0,\  \forall\  i+j>m.
\end{equation}
If $a_{m-k,k}+\cdots+a_{0,m}=0, \ a_{i,m-i}\in H^{i,m-i}(x)$ for some $0\leq k\leq m$ with $a_{m-k,k}\not=0$,  then
$$
-a_{m-k,k}=a_{m-k-1,k+1}+\cdots+a_{0,m}\in F^{m-k}_0\cap\overline{F^{k+1}_0}^x
\Rightarrow a_{k,m-k}=0
$$
which is a contradiction. 
The proof in the other direction is a consequence of 
$$
F^i_0\cap \overline{F^j_0}^x=H^{i,m-i}(x)\cap H^{m-j,j}(x),\ i+j>m.
$$
%---------------------------------------------------------
\subsection{Period domain $U$}
\label{Uspace}
Define
$$
\perdo:=\{x\in \L \mid x \hbox{ satisfies P1 }\},
$$
\index{$\perdo$, period domain}
$$
U:=\{x\in\L\mid x \hbox{ satisfies  P1,P2, P3 } \}.
$$
\index{$U$, period type domain}
\begin{proposition}
The set $\perdo$ is an analytic subset of $\L$ and $U$ is an open 
subset of $\perdo$.
\end{proposition}
\begin{proof}
Take a basis $\omega_i,\ i=1,2,\ldots,h$ of $V_0$ compatible with the 
Hodge filtration. The property (\ref{17oct2011}) is given by
$$
\psi_\C(x)(\omega_r,\omega_s)=0,\ r\leq h^i,\ s\leq h^j,\ i+j>m
$$
and so $\perdo$ is an analytic subset of $\L$.

Now choose a basis $\delta$ of $V_\Z(x)$ and write 
$\delta=p\omega$. Using $\omega$ we may assume that $V_0=\C^h$
and $\delta$ is constituted by the rows of $p$. We have
$$
\omega=p^{-1}\delta\Longrightarrow \overline{\omega}^x=\overline{p}^{-1}\delta=
\overline{p}^{-1}p\omega
$$
Therefore, the rows of $\overline{p}^{-1}p$ are complex conjugates of  
the entries of $\omega$. Now it is easy to verify that if 
the property (\ref{1oct2005}), $\dim(H^{i,m-i}(x))=h^{i,m-i}$ and P3 are valid
for one $x$ then they are valid for all points in a small neighborhood of $x$
(for P3 we may first restrict $\psi_0$ to the product of sphere of radius $1$ and 
center $0\in\C^h$).
\end{proof}
%---------------------------------------------------------------------------
\subsection{An algebraic group}
\label{20sept2010}
 Let $G_0$ be the
algebraic group
$$
G_0:={\rm Aut}(F_0^\bullet, \psi_0):=
$$
$$
\{g: V_0\to V_0 \text{ linear }  \mid g(F_0^i)=F_0^i,\ %\i=0,1,2,\ldots, m+1,\ 
\psi_0(g(\omega_1),g(\omega_2))=\psi_0(\omega_1,\omega_2), \omega_1,\omega_2\in V_0 \}.
$$ 
It acts from the right on $\L$ in a canonical way: 
$$
xg:=g^{-1}(x),\ \psi_\Z(xg)(\cdot,\cdot):=\psi_\Z(g(\cdot),g(\cdot)),\  g\in G_0,\ x\in\L.
$$ 
One can easily see that for all $\omega\in V_0$, $x\in\L$ and $g\in G$
we have
$$
\overline{\omega}^{xg}=g^{-1}\overline{g(\omega)}^x.
$$
\begin{proposition}
\label{oct2005}
The properties P1, P2 and P3 are invariant under the action of $G_0$.
\end{proposition}
\begin{proof}
The property $P1$ for $xg$ follows from the definition.
Let $x\in\L,\ g\in G_0$ and $\omega\in V_0$. We have
\begin{eqnarray*}
H^{i,m-i}(xg) & = &  F^i_0\cap \ovl{F^{m-i}_0}^{xg}=
F^i_0\cap g^{-1}\ovl{g(F^{m-i}_0)}^x= F^i_0\cap g^{-1}(\ovl{F^{m-i}_0}^x) \\
& = & g^{-1}(F^i_0\cap \ovl{F^{m-i}_0}^x)=
g^{-1}(H^{i,m-i}(x))
\end{eqnarray*}
and
$$
\psi_\C(xg)(\omega,\ovl{\omega}^{xg})=\psi_\C(x)(g(\omega),gg^{-1}
\ovl{g(\omega)}^{x})=\psi_\C(x)(g(\omega),\ovl{g(\omega)}^{x}).
$$
These equalities prove the proposition.
\end{proof}
The above proposition implies that $G_0$ acts from the right on $U$. We fix a basis $\omega_i, i=1,2,\ldots,h,$ of $V_0$
compatible with the Hodge filtration $F_0^\bullet$ and, if there is no danger of confusion, we identify each $g\in G_0$ with the $h\times h$ matrix $\tilde g$ given by
\begin{equation}
\label{erica}
[g^{-1}(\omega_1),g^{-1}(\omega_2),\ldots, g^{-1}(\omega_h)]=[\omega_1,\omega_2,\ldots,\omega_h]\tilde g.
\end{equation}
%The space $U/G_0$ is called the moduli of polarized Hodge 
%structures.\index{moduli of polarized Hodge structures}
%------------------------------------------
\subsection{Griffiths period domain}
\label{griffithdomain}
In this section we give the classical approach to the moduli of polarized Hodge structures due
to P. Griffiths. The reader is referred to \cite{kaus04, kaus00} for more developments in this direction. 
 
Let us fix the $\C$-vector space $V_0$ and the Hodge numbers as in \S  \ref{hodgefilsection}.
Let  also $\sfl$\index{$\sfl$, space of filtrations in $V_0$} 
be the space of filtrations (\ref{hodgefiltration}) in $V_0$. In fact, $\sfl$ has a natural
structure of a compact smooth projective variety. We fix the polarized lattice $x_0\in\L$ and define the Griffiths domain
$$
D:=\{F^\bullet\in \sfl \mid  (V_\Z(x_0),\psi_\Z(x_0),F^\bullet) \hbox{ is a polarized Hodge structure }\}.
$$
\index{Griffiths domain}\index{$D$, Griffiths domain}
The group 
$$
\Gamma_\Z:={\rm Aut }(V_\Z(x_0),\psi_\Z(x_0))
$$
 acts on $V_0$ from the right in the  usual way and this gives us an
action of $\Gamma_\Z$ on $D$. The space $\Gamma_\Z\backslash D $ is the moduli space of
polarized Hodge structures. 

%Recall the usual notations of the classifying space of polarized Hodge 
%structures from \cite{kaus00}. Let $D$ be the Griffiths domain and 
%$G_\Z=\Aut(H_0,<\cdot,\cdot>_0)$. We assume that $<\cdot,\cdot>_0$ in a basis
%of $H_0$ has the matrix $\Psi_0$.
\begin{proposition}
There is a canonical isomorphism
$$
\beta: U/G_0\stackrel{\sim}{\rightarrow} \Gamma_\Z\backslash D.
$$
\end{proposition} 
\begin{proof}
We take $x\in U$ and an isomorphism 
$$\gamma: (V_\Z(x),\psi_\Z(x))
\stackrel{\sim}{\rightarrow} 
(V_\Z(x_0),\psi_\Z(x_0)).
$$
 The pushforward of the Hodge filtration $F^\bullet_0$ under this
isomorphism gives us a Hodge filtration on $V_0$  with respect to 
the lattice $V_\Z(x_0)$ and so it gives us a point 
$\beta(x)\in D$. Different choices of $\gamma$ leads us to the action 
of $\Gamma_\Z$ on $\beta(x)$. Therefore, we have a well-defined map
$$
\beta: U\rightarrow \Gamma_\Z\backslash D.
$$
Since $G_0=\Aut(V_0, F_0^\bullet, \psi_0)$, $\beta$ induces the desired isomorphism (it is surjective because for
any polarized Hodge structure $(V_\Z(x_0),\psi_\Z(x_0),F^\bullet)$ we have $V_\Z(x_0)=V_0,\ \psi_\C(x_0)=\psi_0$ and $F^\bullet=g(F_0^\bullet)$ for some $g\in G_0$).
\end{proof}
The Griffiths domain is the moduli space of polarized Hodge structures 
of a fixed type and with a $\Z$-basis in which the polarization has a fixed matrix form.
Our domain $U$ is the moduli space of polarized Hodge structures of a fixed type
and with a $\C$-basis compatible with the Hodge filtration and for which the polarization has a fixed matrix form.

%Since cohomology with integer coefficients is not defined in algebraic geometry
%over an arbitrary field but de Rham cohomology and its Hodge filtration is
%defined, the Griffiths domain does not seem to have an algebraic counterpart 
%but $U$ corresponds to the moduli of smooth projective varieties $X/\C$ with
%certain differential forms on $X$ and certain topological invariants fixed. 
%This arises the hope that $U$ has a natural algebraic structure.

\section{Period map}
In this section we introduce Poincar\'e duals, period matrices and Gauss-Manin connections in the framework of polarized Hodge structures.

%------------------------------------------------------------------------
\subsection{Poincar\'e dual}
\label{poincaredual}
In this section we explain the notion of Poincar\'e dual.
\index{Poincar\'e dual}
Let $(V_\Z(x),\psi_\Z(x))$ be a polarized lattice and 
$\delta\in V_\Z(x)^\vee$, where $\vee$ means
the dual of a $\Z$-module. We will use the symbolic integral notation
\index{$\int_{\delta}\omega$, integral of $\omega$ on $\delta$}
$$
\int_{\delta}\omega:=\delta(\omega),\ \forall \omega\in V_0.
$$ 
The equality
\begin{equation}
\label{29dec08}
\int_{\delta}\overline{\omega}^x=\overline{\int_{\delta}\omega},\ \forall \omega\in V_0,\ \delta\in V_\Z(x)^\vee
\end{equation}
follows directly from the definition. The Poincar\'e dual of $\delta\in V_\Z(x)^\vee$ is an element $\delta^\podu\in V_\Z(x)$ with 
the property  
$$
\int_{\delta} \omega=\psi_\Z(x)(\omega, \delta^\podu),\ \forall \omega\in 
V_\Z(x).
$$
It exists and is unique  because $\psi_\Z$ is non-degenerate. 
Using the Poincar\'e duality one defines the dual polarization \index{dual polarization}
$$
\psi_\Z(x)^\vee(\delta_i,\delta_j):=\psi_\Z(x)(\delta_i^\podu,\delta_j^\podu),\ \delta_i,
\delta_j\in V_\Z(x)^\vee.
$$
We have
$$
(A^\vee \delta)^{\podu}=A^{-1}\delta^{\podu}, \ \forall A\in \Gamma_\Z,\ \delta\in V_\Z(x_0)^\vee, 
$$
where $A^\vee:  V_\Z(x_0)^\vee\rightarrow  V_\Z(x_0)^\vee$ is the induced dual map.  This follows from:
$$
\int_{A^{\vee}\delta}\omega=\int_{\delta}A\omega=\psi_\Z(x_0)(A\omega, \delta^\podu)
=\psi_\Z(x_0)(\omega, A^{-1}\delta^\podu),\ \forall\omega\in V_0.
$$ 
We define
$$
\Gamma_\Z^\vee:=\Aut(V_\Z(x_0)^\vee, \psi_\Z(x_0)^\vee).
$$
It follows that 
$\Gamma_\Z\rightarrow \Gamma_\Z^\vee,\ A\mapsto A^\vee$ is an isomorphism of groups.
%-----------------------------------------------------------------------
\subsection{Period matrix}
\label{22oct04}
Let $\omega_i, i=1,2,\ldots,h$ be a
$\C$-basis of $V_0$ compatible with $F_0^ \bullet$. Recall that $\omega$ means the $h\times 1$ matrix with entries $\omega_i$.
For $x\in U$, we take a $\Z$-basis $\delta_{i},\ i=1,2,\ldots,h$ of $V_\Z(x)^\vee$ such that
the matrix of $\psi_\Z(x)^\vee$ in the basis $\delta$ is $\Psi_0$. 
%For an arbitrary lattice $V_\Z(x)$ with
%$p\in P, \alpha(p)=x$, where $\alpha$ is the map in (\ref{2.5.06}),   we obtain a
%$\Z$-basis $\delta=\delta_x:=p^{\vee}(\delta_{x_0})$,\index{$\delta_x$, a basis} where 
%$p^{\vee}: V_\Z(x_0)^{\vee}\rightarrow V_\Z(x)^{\vee}$ is the map
%induced in the dual spaces. 
We define the abstract period matrix/period map
in the following way:
$$
\per=\per (x)=[\int_{\delta_i}\omega_j]_{h\times h} :=\begin{pmatrix}
\int_{\delta_1} \omega_1 & \int_{\delta_1} \omega_2  
 & \cdots &  \int_{\delta_1} \omega_h  \\
\int_{\delta_2} \omega_1 & \int_{\delta_2} \omega_2 
   & \cdots &  \int_{\delta_2} \omega_h  \\
\vdots  & \vdots  & \vdots  & \vdots  \\ 
\int_{\delta_h} \omega_1 & \int_{\delta_h} \omega_2  
  & \cdots &  \int_{\delta_h} \omega_h
\end{pmatrix}.
$$
Instead of the period matrix it is useful to use the matrix 
$$
\perr=\perr(x),\ \hbox{ where } \delta^\podu=\perr \omega.
$$\index{$\perr$, period type matrix}
Then we have
$$
%(\delta^\podu)^\tr=\omega^\tr \perr^\tr\Longrightarrow 
\Psi_0^\tr=\per\cdot \perr^\tr.
$$
If  we identify $V_0$ with $\C^h$ through the basis 
$\omega$ then $q$ is a matrix whose rows are the entries of $\delta$.
We define  $P$ to be the set of period matrices $\per$. 
We write an element $A$ of $\Gamma_\Z$ in a 
basis of $V_\Z(x_0)$, and redefine 
$\Gamma_\Z$:
$$
\Gamma_\Z:=\{A\in \GL(h,\Z)\mid  A\Psi_0A^\tr=\Psi_0\}.
$$
The group $\Gamma_\Z$ acts on $P$ from
the left by the usual multiplication of matrices and
$$
U=\Gamma_\Z\backslash P.
$$
In a similar way, if we identity each element $g$ of $G_0$ with  the matrix $\tilde g$ in (\ref{erica}) then
$G_0$ acts from the right on $P$ by the usual multiplication of matrices.

%%%%%%%%%%%%%%%%%%%%%%%%%%%%%%%%%%%%%%%%%%%%%
%--------------------------------------------------------------------------
\subsection{A canonical connection on $\L$} 
\label{connectiononL}
%Recall the terminology related to connections in Chapter \ref{foliation}.
We consider the trivial bundle $\Hb=\L\times V_0$ on $\L$. On $\Hb$ we have a 
well-defined integrable connection
$$
\nabla: \Hb\rightarrow \Omega^1_\L\otimes_{\O_\L} \Hb
$$
such that a section $s$ of $\Hb$ in a small open set 
$V\subset\L$ with the property 
$$
s(x)\in \{x\}\times V_\Z(x),\ x\in V.
$$
is flat. Let $\omega_1,\omega_2,\ldots,\omega_h$ be a basis of $V_0$ compatible with the Hodge filtration $F_0^\bullet$.
We can consider $\omega_i$ as a global section of $\Hb$ and so we have
\begin{equation}
\nabla\omega=A\otimes\omega,\ 
A=\left ( \begin{array}{cccc}    
            \omega_{11}& \omega_{12} & \cdots & \omega_{1h} \\
            \omega_{21}& \omega_{22} & \cdots & \omega_{2h} \\
            \vdots     & \vdots      & \ddots &  \vdots    \\
            \omega_{h1}& \omega_{h2} & \cdots & \omega_{hh} \\
       \end{array} \right ),\ \omega_{ij}\in H^0(\L, \Omega_\L^1).
\end{equation}
%We extend the definition of $\nabla$
%$$
%\nabla:\Omega^p_\L\otimes\Hb\rightarrow \Omega^{p+1}_\L\otimes\Hb,\ 
%\nabla(\alpha\otimes \omega)=d\alpha\otimes\omega+(-1)^p
%\alpha \wedge \nabla\omega.
%$$
\index{$\nabla$, connection}
$A$ is called the connection matrix of $\nabla$ in the basis $\omega$.\index{connection matrix} 
The connection $\nabla$ is integrable and so $dA=A\wedge A$:
%\begin{eqnarray*}
%\nabla^2(\omega_i) & = & 
%\nabla(\sum_{j=1}^{h}\omega_{ij}\otimes \omega_j)=
%\sum_{j=1}^{h}d\omega_{ij}\otimes \omega_j-
%\omega_{ij}\wedge \nabla(\omega_{j}) \\
%& =& 
%\sum_{j=1}^h d\omega_{ij}\otimes \omega_j-
%\sum_{j=1}^{h}\omega_{ij}\wedge (\sum_{k=1}^h\omega_{jk}\otimes\omega_k)
 %\\
% & = &
%\sum_{j=1}^h (d\omega_{ij}-\sum_{k=1}^{h}\omega_{ik}\wedge \omega_{kj})\otimes
%\omega_j=0.
%\end{eqnarray*}
%This implies that
\begin{equation}
\label{14oct05}
d\omega_{ij}=\sum_{k=1}^{h} \omega_{ik}\wedge \omega_{kj},\ i,j=1,2,\ldots,h.
\end{equation}
Let $\delta$ be a basis of flat sections. Write $\delta=\perr\omega$.
We have
$$
\omega=\perr^{-1}\delta\Rightarrow
\nabla(\omega)=d(\perr^{-1})\perr \omega
\Rightarrow
$$
$$
A=d\perr^{-1}\cdot\perr=d(\per^\tr\cdot\Psi_0^{-\tr})\cdot(\Psi_0^{\tr}\cdot 
\per^{-\tr})=d(\per^{\tr})\cdot\per^{-\tr}.
$$
and so
\begin{equation}
\label{omidshoot}
A=d(\per^{\tr})\cdot\per^{-\tr}.
\end{equation}
where $\per$ is the abstract period map.
We have used the equality $\Psi_0=\per\cdot \perr^\tr$. Note that the entries of $A$ are holomorphic
$1$-forms on $\L$ and a fundamental system for the linear differential equation 
$dY=A\cdot Y$
in $\L$
is given by $Y=\per^{\tr}$:
$$
d\per^\tr=A\cdot \per^\tr.
$$
We define the Griffiths transversality distribution by:
 \begin{equation}
\F_{gr}:
\label{grtr1}
\omega_{ij}=0,\ i\leq h^{m-x},\ j> h^{m-x-1},\ x=0,1,\ldots,m-2.
\end{equation}
A holomorphic map $f:V\to U$, where $V$ is an analytic variety, is called a period map if it is tangent to the Griffiths transversality distribution, that is, for all $\omega_{ij}$ as in
 (\ref{grtr1}) we have $f^{-1}\omega_{ij}=0$. 
%------------------------------------------------------------------------
\subsection{Some functions on $\L$}

For two vectors $\omega_1,\omega_2\in V_0$, we have the following holomorphic function on $\L$:
$$
\L\rightarrow \C, \ x\mapsto \psi_\C(x)(\omega_1,\omega_2).
$$
We  choose a basis
$\omega$ of $V_0$ and $\delta$ of $V_\Z(x)^\vee$ for $x\in\L$  and write $\delta^\podu=\perr \cdot \omega$. 
Then 
\begin{equation}
\label{10oct05}
F:=[\psi_\C(x)(\omega_i,\omega_j)]=(\perr^{-1})\Psi_0\perr^{-\tr}= \per^\tr\Psi_0^{-\tr}\per
\end{equation}
(we have used the identity $\Psi_0^\tr=\per\cdot \perr^\tr$).
The matrix $F$
satisfies the differential equation
\begin{equation}
\label{7nov05}
dF=A\cdot F+F\cdot A^\tr,
\end{equation}
where $A$ is the connection matrix.
The proof is a straightforward consequence of (\ref{10oct05}) and (\ref{omidshoot}):
\begin{eqnarray*}
dF &=& d(\per^\tr\Psi_0^{-\tr}\per)\\
&=& (d\per^\tr)\Psi_0^{-\tr}\per+\per^\tr\Psi_0^{-\tr}(d\per)\\
&=& A\cdot F+F\cdot A^\tr
\end{eqnarray*}
It is easy to check that every solution of the differential equation (\ref{7nov05})
is of the form $\per^\tr\cdot C\cdot \per$ for some constant $h\times h$ matrix $C$ with
entries in $\C$ (if $F$ is a solution of (\ref{7nov05}) then $F\cdot \per^{-1}$ is a solution
of $dY=A\cdot Y$). We restrict $F, A$ and $\per$ to $U$ and we conclude that
\begin{equation}
 \label{12sep2010}
\Phi_0=\per^{\tr}\Psi_0^{-\tr}\per
\end{equation}
$$
A\cdot \Phi_0=-\Phi_0\cdot A^\tr.
$$
where by definition $F|_U$ is the constant matrix $\Phi_0$.  
% We fix an isomorphism of $\C$-vector spaces $o:\wedge^h V_0\cong \C$. It is 
% called an orientation. Now, we have the determinant map
% $$
% \det:(V_0)^h\rightarrow \C, \ \det(\omega_1,\omega_2,\ldots,
% \omega_h):=o(\omega_1\wedge \omega_2\wedge\cdots\wedge \omega_h).
% $$
% Using $\det$, one can define:
% $$
% {\det}^2:\L\rightarrow\C,\ {\det}^2(x):=
% \det(\delta_1,\delta_2,\cdots,\delta_h)^2=\det(\perr)^2=
% \frac{\det(\Psi_0)^2}{\det(\per)^2},
% $$
% where $\delta:=(\delta_i)_{i\in I}$ is a $\Z$-basis of $V_\Z(x)$ for 
% which the bilinear form $\psi_\Z(x)$ has the form $\Psi_0$. Taking another
% basis will contribute $\det(A),\ A\in \Gamma_\Z$, which is $+1$ or $-1$, to 
% the $\det$ function and so $\det^2$ is a well-defined function.
% In case $\det(A)=1$ for all $A\in \Gamma_\Z$, we can define the $\det$ 
% function in $\L$.

We have a plenty of non-holomorphic functions on $\L$. 
 For two elements 
$\omega_1,\omega_2\in V_0$ we define
$$
\L\rightarrow \C,\ x\mapsto \psi_\C(x)
(\omega_1, 
\overline{\omega_2}^x).
$$
Let $\omega$ and $\delta$ be as before. We write
$\delta^{\podu}=\overline{\perr}\cdot\overline{\omega}^{x}$ and we have
\begin{equation}
\label{10.10.05}
G:=[\psi_\C(x)(\omega_i,\bar \omega_j^x)]=\per^\tr\Psi_0^{-\tr}\overline{\per}=(\perr^{-1})\Psi_0\overline{\perr}^{-\tr}
\end{equation}
The matrix $G$ satisfies the differential equation
\begin{equation}
\label{7no05}
dG=A\cdot G+G\cdot \overline {A}^\tr,
\end{equation}
where $A$ is the connection matrix.

% For $\omega\in V_0,\ x\in \L$  and 
% $\epsilon\in\{0,1\}$ define $\overline{\overline \omega}^{x,\epsilon}=
% \omega$ if 
% $\epsilon=0$ and
% $=\overline{\omega}^x$ otherwise. 
% Let $\epsilon:I\rightarrow \{0,1\}$ 
% be a function and $\omega=(\omega_i)_{i\in I}$ be $h$ elements in $V_0$.
% We have the following complex valued analytic function on $\L$:
% $$
% f_\omega^\epsilon:\L\rightarrow \C,\ f_\omega^\epsilon(x)=
% \det(\overline{\overline \omega}^{x,\epsilon}).
% $$
% \index{$f_\omega^\epsilon$, a function on $\L$}

%%%%%%%%%%%%%%%%%%%%%%%%%%%%%%%%%%%%%%%%%%%%%
%%%%%%%%%%%%%%%%%%%%%%%%%%%%%%%%%%%%%%%%%%%%%%%%%%%%%%%%%%%%%%%55
\section{Quasi-modular forms attached to Hodge structures}
In this section we explain what is a quasi-modular form attached to a given fixed data of Hodge structures and a full family
of enhanced projective varieties.

\subsection{Enhanced projective varieties}
\label{batoocarro}
Let $X$ be a complex smooth projective  variety of a fixed topological type. This means that we fix a $C^\infty$ manifold  $X_0$ and assume that $X$ as a $C^\infty$-manifold 
is isomorphic to $X_0$ (we do not fix the isomorphism).  
Let $n$ be the complex dimension of $X$ and let $m$ be an integer with $1\leq m\leq n$. We fix an element  $\theta\in H^{2n-2m}(X,\Z)\cap H^{n-m,n-m}(X)$. By $H^i(X,\Z)$ we mean its image 
in $H^i(X,\C)=H^i_\dR(X)$; therefore, we have killed the torsion. We consider the bilinear map
$$
\langle \cdot,\cdot \rangle_\C : H_\dR^{m}(X)\times H_\dR^{m}(X)\to \C, \ \langle \omega,\alpha\rangle= 
\frac{1}{(2\pi i)^m}\int_X \omega \cup \alpha\cup \theta.
$$
The $(2\pi i)^{-m}$ factor in the above definition ensures us that the bilinear map 
$\langle \cdot,\cdot \rangle_\C$ is defined for the algebraic de Rham cohomology (see for instance Deligne's lecture in \cite{dmos}). 
We
assume that it is non-degenerate. 
%, symmetric if $m$ is even and skew if $m$ is 
%odd.
The cohomology $H^m_\dR(X)$  
is equipped with the so-called Hodge filtration $F^\bullet$. We assume that the Hodge numbers  $h^{i,m-i},\ i=0,1,2,\ldots,m$ coincide with those fixed in this article. 
We consider Hodge structures with an isomorphism
$$
(H^m_\dR(X),F^\bullet, \langle \cdot,\cdot \rangle_\C)\cong (V_0, F^\bullet_0, \psi_0).
$$ 
From now on, by an enhanced projective variety we mean all the data described in the previous paragraph. 

We also need to introduce families of enhanced projective varieties. Let $V$ be an irreducible affine variety and  $\O_V$ be the ring of regular functions on $V$. By definition $V$ is the underlying complex space of ${\rm Spec(\O_V)}$ and 
$\O_V$ is a finitely generated reduced $\C$-algebra without zero divisors. Also, let $X\to V$ be a family of smooth projective varieties as in the previous paragraph. We will also use the notations  
$\{X_t\}_{t\in V}$ or
$X/V$ to denote $X\to V$.   
The de Rham cohomology $H^m_\dR(X/V)$ and its 
Hodge filtration $F^\bullet H^m_\dR(X/V)$ are $\O_V$-modules (see for instance \cite{gro66}) and in a similar way we have 
$\langle \cdot,\cdot \rangle_{\O_V} : H_\dR^{m}(X/V)\times H_\dR^{m}(X/V)\to \O_V$.
Note that we fix an element $\theta\in F^{n-m}H^{2n-2m}_\dR(X/V)$ and assume that it induces in each fiber $X_t$ an element in $H^{2n-2m}(X_t,\Z)$.
 We say that the family is 
enhanced if we have an isomorphism
\begin{equation}
 \label{15sep2010}
\left (H^m_\dR(X/V),\  F^\bullet H^m_\dR(X/V),\  \langle\cdot,\cdot\rangle_{\O_V}\right ) \cong \left(V_0\otimes_\C {\O_V},\  F_0^\bullet\otimes_\C \O_V, \ \psi_0\otimes_\C
\O_V\right ).
\end{equation}
We fix a basis $\omega_i,\ i=1,2,\ldots,h$ of $V_0$ compatible with the filtration $F_0^\bullet$. Under the above isomorphism we get a basis $\tilde\omega_i, \ i=1,2,\ldots,h$ of the $\O_V$-module 
$H^m_\dR(X/V)$ which is compatible with the Hodge filtration and the bilinear map $\langle\cdot,\cdot\rangle_{\O_V}$ written in this basis is a constant matrix. 
This gives us another formulation of
an enhanced family of projective varieties.
An enhanced family of projective varieties $\{X_t\}_{t\in V}$ is full if we have an algebraic action of $G_0$ (defined in \S\ref{20sept2010}) from the right  on $V$ (and hence on $\O_V$) such 
that it is compatible with the isomorphism 
(\ref{15sep2010}). This is equivalent to saying  that for $X_t$ and $\tilde\omega_i, \ i=1,2,\ldots,h$ as above, we have an isomorphism %$f:X_{t\bullet g}\to X_t$ such that 
$$
(X_{tg},[\tilde \omega_1,\tilde\omega_2,\ldots,\tilde\omega_h])\cong (X_t,[\tilde \omega_1,\tilde\omega_2,\ldots,\tilde\omega_h]g), \ t\in V, \ g\in G_0,
$$
(recall the matrix form of $g\in G_0$ in (\ref{erica})).
A morphism $Y/W\to X/V$ of two families of enhanced projective varieties is a commutative diagram
$$
\begin{matrix}
Y &\to& X\\
\downarrow & &\downarrow\\
W &\to& V
\end{matrix}
$$
such that 
$$
\begin{matrix}
H^m(X/V)&\to& H^m(Y/W)\\
\downarrow& &\downarrow\\
V_0\otimes_\C \O_V&\to& V_0\otimes_\C \O_W
\end{matrix}
$$
is also commutative. 
%%--------------------------------------------
\subsection{Period map}
\label{11sep06}
For an enhanced projective variety $X$, we consider the image of $H^m(X,\Z)$ in $H^m(X,\C)\cong H^m_\dR(X)\cong V_0$ and hence  
we obtain a unique point in $U$. Note that by this process we  kill  torsion elements in $H^m(X,\Z)$.   
We fix bases $\omega_i$ and $\tilde\omega_i$ as in \S\ref{batoocarro} and a basis $\delta_i, \ i=1,2,\ldots,h$ of $H_m(X,\Z)=H^m(X,\Z)^\vee$ with $[\langle \delta_i,\delta_j\rangle]=\Psi_0$ and we see that the corresponding 
point in $U:=\Gamma_\Z\backslash P$ is given by the equivalence class of the geometric period matrix $[\int_{\delta_i}\tilde \omega_j]$.

%See Deligne's article in \cite{dmos} for the reason why we insert the factor $\frac{1}{(2\pi i)^{\frac{m}{2}}}$.  
For any family of enhanced projective varieties $\{X_t\}_{t\in V}$ we get
$$
\per: V\to U
$$
which is holomorphic. It satisfies the so-called 
Griffiths transversality, that is, it is tangent to the Griffiths transversality distribution.
 It is called a geometric period map. 
The pullback of the connection $\nabla$ constructed in \S\ref{connectiononL} by the period map $\per$ is 
the Gauss-Manin connection of the family $\{X_t\}_{t\in V}$. If the family is full then the geometric period map commutes with the action of $G_0$:
$$
\per(tg)=\per(t)g,\ g\in G_0, \ t\in V.
$$
%---------------------------------------------------------------------------
\subsection{Quasi-modular forms}
\label{13sep2010}
Let $M$ be the set of enhanced projective varieties with  
the fixed topological data explained in \S\ref{batoocarro}. 
We would like to prove that $M$ is in fact an affine variety. The first step in
developing a quasi-modular form theory attached to enhanced projective varieties is to solve the following conjectures. 
%Recall that for an enhanced projective variety 
%we have fixed the topological data explained in \S\ref{batoocarro}. 
% and even more it is inside another affine variety which 
%includes degenerations of enhanced projective varieties. We formulate the following problem. 
\begin{conjecture}\rm
There is an affine variety $T$ and a full family $X/T$ of enhanced projective varieties which is universal in the following sense: 
for any family of enhanced projective varieties $Y/S$ we have a unique morphism of $Y/S\to X/T$ of enhanced 
projective varieties.
\end{conjecture}
We would also like to find a universal family which describes the degeneration of projective varieties:
\begin{conjecture}\rm
There is an affine variety $\tilde T\supset T$ of the same dimension as $T$ and  with the following property: 
for any family $f: Y\to S$ of projective varieties with fixed prescribed topological data, but not necessarily 
enhanced and smooth,   and with the discriminant variety $\Delta\subset S$,  the map 
$Y\backslash f^{-1}(\Delta)\to S\backslash \Delta$ is an underlying morphism of an enhanced family, and hence, we have 
the map $S\backslash \Delta\to T$ which extends to $S\to \tilde T$. The conjecture is about the existence of $\tilde T$ with
such an extension property. 
%Moreover, we can construct in an effective way affine coordinates for $T$. 
\end{conjecture}
Similar to Shimura varieties, we expect that $T$ and $\tilde T$ are affine varieties defined over $\bar\Q$. 
Both conjectures are true in the case of elliptic curves (see the discussion in the Introduction). In this case, the function ring of $T$ (resp. $\tilde T$)
 is  
$\C[t_1,t_2,t_3,\frac{1}{27t_3^2-t_2^3}]$ (resp.  $\C[t_1,t_2,t_3]$ ). We have also verified the conjectures for a particular class of Calabi-Yau varieties (see \S\ref{1111} and \cite{ho10*}).

Now, consider the case in which both conjectures are true. We are going to explain the rough idea of the algebra of quasi-modular  forms attached to all fixed data that we had. It
 is the pullback of 
the $\C$-algebra of regular functions in $\tilde T$ by the composition
\begin{equation}
\label{khoshgel??}
\uhp \stackrel{i}{\hookrightarrow} P|_{\Im(\per)}\to U\mid_{\Im(\per)}\stackrel{\per^{-1}}{\to} T\hookrightarrow\tilde T.
\end{equation}
Here $\per$ is the geometric period map.
We need that the period map is 
locally injective (local Torelli problem) and hence $\per^{-1}$ is a local inverse map. The set $\uhp$ is a subset of the set
of period matrices $P$ and it will play the role of the Poincar\'e upper half-plane. If the Griffiths period domain $D$ is  Hermitian symmetric then it is biholomorphic to $D$ 
(see \ref{h10=h01}); however, in other cases it depends on the universal period map $T\to U$ and its dimension is the dimension of the deformation space of the projective 
variety. In this case we do not need
to define $\uhp$ explicitly (see \ref{1111}).  
More details of this discussion will be explained by two examples of the next section. 
%%%%%%%%%%%%%%%%%%%%%%%%%%%%%%%%%%%%%%%%%%%%%%%%%%%%%%%%%%%%%%%%%%%%%%%%%%%%%%%%%%%%%%%%%%%%%%%%%%%%%%%%%%%%5
\section{Examples}
In this section we discuss two examples of Hodge structures and the corresponding quasi-modular form algebras: 
those attached to  mirror quintic Calabi-Yau varieties and principally 
polarized Abelian varieties. The details of the first case are done in \cite{ho10*, ho11} and we will sketch the results which are related to 
the main thread of the present text. For the second
case there is much work that has been done  and I only sketch some ideas. Much of the work for K3 surfaces endowed with polarizations 
has  been already done by many authors, see \cite{do10-1} and the references therein. The generalization of such  results 
to Siegel quasi-modular forms is work for the future.    
%------------------------------------------------------------------------
\subsection{Siegel quasi-modular forms} %$m=1, h^{10}=h^{01}=g$}
\label{h10=h01}
We consider the case in which the weight $m$ is equal to $1$
and the polarization matrix is:
$$
\Psi_0=\mat 0{-I_g}{I_g}0,
$$ 
where $I_g$ is the $g\times g$ identity matrix. 
In this case $g:=h^{10}=h^{01}$ and $h=2g$. 
We take a basis $\omega_i,\ i=1,2,\ldots,2g,$
%$(\omega^1_1,\omega^1_2,\ldots,\omega^1_g,  \omega^2_1,\omega^2_2,\ldots,\omega^2_g )$ 
of $V_0$
compatible with $F^\bullet_0$, that is,  the first $g$ elements form a basis of $F^1_0$. We further assume that 
the polarization $\psi_0: V_0\times V_0\to \C$ in the basis $\omega$ has the form $\Phi_0:=\Psi_0$. 
Because of the particular format of $\Psi_0$,  both these assumptions do not contradict each other.  
%We fix an orientation $o:\wedge^h V_0\cong \C$ and assume that  such a basis is positively 
%oriented, i.e. 
%%$o(\omega^1_1\wedge \omega^1_2\wedge\ldots\wedge\omega^1_g\wedge\omega^2_1\wedge\omega^2_2\wedge\ldots\wedge\omega^2_g )=1$. 
%$o(\omega_1\wedge \omega_2\wedge\cdots\wedge \omega_{2g})=1$. 
We take
a basis $\delta$ of $V_\Z(x)^\vee$ such that the intersection form in this basis is of the form $\Psi_0$ and we write the associated period matrix 
in the form
$$
[\int_{\delta_i}\omega_j]=\mat {x_1}{x_2}{x_3}{x_4},
$$ 
where $x_i,i=1,\ldots,4,$ are
$g\times g$ matrices. 
Since 
$\Psi_0^{-\tr}=\Psi_0$, 
%and $\Psi_0=\per^\tr \Psi_0^{-\tr}\per$ (which follows from property P1) we get the following polynomial relations between periods 
we have 
\begin{eqnarray*}
\mat 0{-I_g}{I_g}0 &=&
%\Psi_0=[\langle \omega_i,\omega_j\rangle ]=
%\mat {f_{\omega_1,\omega_1}}
%     {f_{\omega_1,\omega_2}}
%     {f_{\omega_2,\omega_1}}
%     {f_{\omega_2,\omega_2}}=
\mat {x_1^\tr}{x_3^\tr}{x_2^\tr}{x_4^\tr}
\mat 0{-I_g}{I_g}0 
%\Psi_0
\mat {x_1}{x_2}{x_3}{x_4} \\
&=&
\mat
{x_3^\tr x_1-x_1^\tr x_3}
{x_3^\tr x_2-x_1^\tr x_4}
{x_4^\tr x_1-x_2^\tr x_3}
{x_4^\tr x_2-x_2^\tr x_4}
\end{eqnarray*}
and 
\begin{eqnarray*}
%\mat {f_{\omega_1,\bar\omega_1}}
%     {f_{\omega_1,\bar\omega_2}}
%     {f_{\omega_2,\bar\omega_1}}
%     {f_{\omega_2,\bar\omega_2}}=
[\langle \omega_i,\bar\omega_j^x\rangle ] &=&
\mat {x_1^\tr}{x_3^\tr}{x_2^\tr}{x_4^\tr}\mat 0{-I_g}{I_g}0 
\mat {\bar x_1}{\bar x_2}{\bar x_3}{\bar x_4}
\\
&=&
\mat
{x_3^\tr \bar x_1-x_1^\tr \bar x_3}
{x_3^\tr \bar x_2-x_1^\tr \bar x_4}
{x_4^\tr \bar x_1-x_2^\tr \bar x_3}
{x_4^\tr \bar x_2-x_2^\tr \bar x_4}.
\end{eqnarray*}
The properties P1, P2 and  P3 are summarized in the properties 
$$
x_3^\tr x_1=x_1^\tr x_3,\ x_3^\tr x_2-x_1^\tr x_4=-I_g,
$$
$$
x_1,x_2\in \GL (g,\C),
$$
$$
\sqrt{-1}(x_3^\tr \bar x_1-x_1^\tr \bar x_3)\text{ is a 
positive matrix}.
$$
 By definition $P$ is the set of all $2g\times 2g$ matrices $\mat {x_1}{x_2}{x_3}{x_4}$ satisfying the above properties:
The matrix $x:=x_1x_3^{-1}$ is well-defined  and invertible and
satisfies the well-known Riemann relations:\index{Riemann relations}
$$
x^\tr=x,\ \Im(x) \hbox{ is a positive matrix}.
$$
The set of matrices $x\in \Mat^{g\times g}(\C)$ with the above properties is
called the Siegel upper half-space and is denoted by $\uhp$. \index{Siegel upper half-plane}
\index{$\uhp_g$, Siegel upper half-plane of dimension $g$}
We have $U=\Gamma_\Z\backslash P$, where 
$$
\Gamma_\Z={\mathrm Sp}(2g, \Z)=
\{\mat abcd\in \GL (2g,\Z) \mid ab^\tr =ba^\tr ,\ cd^\tr=dc^\tr,\ ad^\tr-bc^\tr=I_g \}.
$$
We have also
$$
G_0=\{\mat{k}{k'}{0}{k^{-\tr}}\in \GL(2g,\C)\mid k{k'}^\tr=k'k^\tr\}
$$
which acts on $P$ from the right. The group ${\mathrm Sp}(2g, \Z)$ acts on $\uhp$ by
$$
\mat abcd\cdot  x=(ax+b)(cx+d)^{-1},\ \mat abcd \in {\mathrm Sp}(2g, \Z),\ x\in\uhp
$$ 
and we have the isomorphism
$$
U/G_0 \rightarrow {\mathrm Sp}(2g, \Z)\backslash \uhp, 
$$
given by 
$$
\mat{x_1}{x_2}{x_3}{x_4}\rightarrow x_1x_3^{-1}.
$$
To each point $x$ of  $P$ we associate a triple $(A_x,\theta_x,\alpha_x)$ as follows: We have $A_x:=\C^{g}/\Lambda_x$, where $\Lambda_x$
is the  $\Z$-submodule of $\C^g$ generated by the rows of $x_1$ and $x_3$. We have cycles $\delta_i\in H_1(A_x,\Z),\ i=1,2,\ldots,2g,$ which are defined by the property
$[\int_{\delta_i}dz_{j}]=\matt {x_1}{x_3}$, where $z_j,\ j=1,2,\ldots, g,$ are linear coordinates of $\C^g$.  There is a basis 
$\alpha_x=\{\alpha_1,\alpha_2,\ldots,\alpha_{2g}\}$ of $H_\dR^1(A_x)$ such that
$$
[\int_{\delta_i}\alpha_j]=\mat{x_1}{x_2}{x_3}{x_4}.
$$
The polarization in $H_1(A_x,\Z)\cong \Lambda_x$ (which is defined by  $[\langle \delta_i,\delta_j\rangle ]=\Psi_0$) is an element $\theta_x\in H^2(A_x,\Z)=\bigwedge_{i=1}^2{\rm Hom}(\Lambda_x, \Z)$. 
It gives the following bilinear map 
$$
\langle\cdot,\cdot\rangle : H_\dR^1(A_x)\times H_\dR^1(A_x)\to\C,\ \langle \alpha,\beta \rangle=\frac{1}{2\pi i}\int_{A_x} \alpha\cup \beta \cup \theta_x^{g-1}
$$
which satisfies $[\langle \alpha_i,\alpha_j\rangle]=\Psi_0$.

The triple $(A_x,\theta_x, \alpha_x)$ that we constructed in the previous paragraph does not depend on the action of ${\mathrm Sp}(2g, \Z)$ from the left on $P$; therefore, for each $x\in U$ 
we have constructed such a triple. In fact $U$ is the moduli space of the triples $(A,\theta, \alpha)$ such that $A$ is a principally polarized abelian variety with a polarization $\theta$ and 
$\alpha$  is a basis of $H_\dR^1(A)$ compatible with the Hodge filtration $F^1\subset F^0= H_\dR^1(A)$  and such that $[\langle \alpha_i,\alpha_j\rangle]=\Psi_0$.

We constructed the moduli space $U$ in the framework of complex geometry. In order to introduce Siegel quasi-modular forms, we have 
to study the same moduli space in the framework of algebraic geometry. We have to construct an algebraic  variety $T$ over $\C$ such that the points 
of $T$ are in one to one correspondence 
with the equivalence classes of the triples $(A,\theta,\alpha)$. We also expect that $T$ is an affine variety and it lies inside another affine variety $\tilde T$ which describes the degeneration
of varieties (as it is explained in \S \ref{13sep2010}). 
The pullback of the $\C$-algebra of regular functions on $\tilde T$ through the composition
$$
\uhp\to P\to U\stackrel{\per^{-1}}{\to} T\hookrightarrow \tilde T
$$
is, by definition, the  $\C$-algebra of Siegel quasi-modular forms.
The first map is given by 
$$
z \to \mat {z}{-I_g}{I_g}{0}
$$ 
and the second is the canonical map. The period map in this case is a biholomorphism.
If we impose a functional property for $f$ regarding the action of $G_0$ then
this will be translated into a functional property of a Siegel quasi-modular form with respect to the action of ${\mathrm Sp}(2g, \Z)$. In this way we can even define a 
Siegel quasi-modular form defined over $\bar\Q$ (recall that 
we expect $\tilde T$ to be defined over $\bar\Q$).  
It is left to the reader to verify that the $\C$-algebra of 
Siegel quasi-modular forms 
%contains the classical Siegel modular forms and it 
is closed under derivations with respect to $z_{ij}$ with $z=[z_{ij}]\in\uhp$.
For the realization of all these in the case of 
elliptic curves, $g=1$, see the Introduction and \cite{ho06-2}. See the books  \cite{kl90,fr83, ma71} for more information on Siegel modular 
forms.
%%%%%%%%%%%%%%%%%%%%%%%%%%%%%%%%%%%%%%%%%%%%%%%%%%%%%%%%%%%%%%%%%%%%%%%%%%%%%%%%%%%%%%%%%%%%%5
%----------------------------------------
\subsection{Hodge numbers, 1,1,1,1}
\label{1111}
In this section we consider the case $m=3$ and the Hodge numbers
$h^{30}=h^{21}=h^{12}=h^{03}=1,\ h=4.$
The polarization matrix written in an integral basis is given by
$$
\Psi_0=\begin{pmatrix}
 0&0&1&0\\
0&0&0&1\\
-1&0&0&0\\
0&-1&0&0
\end{pmatrix}.
$$
Let us fix a basis $\omega_1,\omega_2,\omega_3,\omega_4$ of $V_0$ compatible with the Hodge filtration $F_0^\bullet$, a basis $\delta_1,\delta_2,\delta_3,\delta_4\in V_\Z(x)^\vee$ with the 
intersection matrix $\Psi_0$ and let us write  the period matrix in the form $\per(x)=[x_{ij}]_{i,j=1,2,\ldots,4}$. 
We assume that  the polarization $\psi_0$ in the basis $\omega_i$ is given by the matrix
$$
%[\psi_\C(x)(\omega_i,\omega_j)
\Phi_0:=
\begin{pmatrix}
 0&0&0&1\\
0&0&1&0\\
0&-1&0&0\\
-1&0&0&0
\end{pmatrix}.
$$
The algebraic group $G_0$ is defined to be
$$
G_0:=\left \{g=\begin{pmatrix}
 g_{11}&g_{12}&g_{13}&g_{14}\\
0&g_{22}&g_{23}&g_{24}\\
0&0&g_{33}&g_{34}\\
0&0&0&g_{44}
\end{pmatrix}, \  g^\tr \Phi_0g=\Phi_0,\  g_{ij}\in\C\right \}.
$$
One can verify that it is generated by six one-dimensional subgroups, two of them isomorphic to the multiplicative group $\C^*$ and four
of them isomorphic to the additive group $\C$. Therefore, $G_0$ is of dimension $6$.   
We consider the subset $\tilde \uhp$ of $P$ consisting of matrices
\begin{equation}
 \tau=\label{UERJ?}
\begin{pmatrix}
 \tau_0 &1&0&0\\
1&0&0&0\\
\tau_{1}& \tau_{3}& 1&0\\
\tau_{2} & -\tau_0 \tau_{3}+\tau_{1} & -\tau_0& 1 
\end{pmatrix},
\end{equation}
where $\tau_i,\ i=0,1,2,3,$ are some variables in $\C$ (they are coordinates of the corresponding 
moduli space of polarized Hodge structures and so this moduli space is of dimension four). 
The particular expressions for the $(4,2)$ and $(4,3)$ entries of the above matrix follow 
from the polynomial relations (\ref{12sep2010}) between periods. 
The connection matrix $A$ restricted to $\tilde \uhp$ is
$$
d\tau^\tr\cdot \tau^{-\tr}=
\begin{pmatrix}
0&d\tau_0& -\tau_{3}d\tau_0+d\tau_{1} &  -\tau_{1}d\tau_0+\tau_0 d\tau_{1}+d\tau_{2}\\  
0&0&d\tau_{3}&-\tau_{3}d\tau_0+d\tau_{1} \\ 
0&0&       0&       -d\tau_0\\
0&0&       0&       0 
\end{pmatrix}.        
$$
The Griffiths transversality distribution is given by 
 $$
-\tau_{3}d\tau_0+d\tau_{1}=0,\  \ -\tau_{1}d\tau_0+\tau_0 d\tau_{1}+d\tau_{2}=0.
$$
and so, if we consider $\tau_0$ as an independent parameter defined in a neighborhood of $+\sqrt{-1}\infty$, 
and all other quantities 
$\tau_i$ depending on $\tau_0$, then we have
\begin{equation}
\label{retrovisor}
\tau_{3}=\frac{\partial \tau_1}{\partial \tau_0},\ \frac{\partial\tau_2}{\partial \tau_0}=\tau_{1}-\tau_0 \frac{\partial\tau_1}{\partial \tau_0}. 
\end{equation}
In \cite{ho10*} we have checked the conjectures in \S \ref{13sep2010} for the Calabi-Yau  threefolds of mirror 
quintic type. In this case $\dim(T)=7=1+6$, where $1$ is the dimension of the moduli space of mirror quintic Calabi-Yau
varieties and $6$ is the dimension of the algebraic group $G_0$. Hence, we have 
constructed an algebra generated by seven functions in $\tau_0$, which we call it the algebra of quasi-modular forms attached to mirror quintic Calabi-Yau varieties. 
The image of the geometric period map lies in $\uhp$
with
\begin{equation}
\label{irriteiocara}
 \tau_1=-\frac{25}{12}+\frac{5}{2}\tau_0(\tau_0+1)+
\frac{1}{(2\pi i)^2}\sum_{n=1}^\infty \left (\sum_{d|n}n_d d^3\right )\frac{e^{2\pi i \tau_0 n}}{n^2}.
\end{equation}
Here, $n_d$'s are instanton numbers and  the second derivative of $\tau_1$  with respect to $\tau_0$ is 
the Yukawa coupling. The Yukawa coupling itself turns out to be a quasi-modular form in our context but not its double primitive $\tau_1$.
 The set 
$\uhp$ is a subset of $\tilde \uhp$ defined by (\ref{retrovisor}) and (\ref{irriteiocara}).
As far as I know  this is the first case in which the Griffiths period domain is not Hermitian symmetric  
and we have an attached algebra of quasi-modular forms and even the Global Torelli problem is true; that is, the 
period map is globally injective (see \cite{gri09}). However, note that in \cite{ho10*} we have only used the local injectivity of the period map. 
In this case we can prove that the pullback map from the algebra of regular functions on $\tilde T$ to the algebra of holomorphic functions on $\uhp$ is injective.
%Torelli problem is true, i.e the period map is a one to one map 
%(see \cite{kaus04, gri09}). This ensurues us that, at least,  the image of the generalized period map has a canonical algebraic structure given by the algebraic structure of the domain. 
%However, the construction of the corresponding quasi automorphic functions is done recently in \cite{ho10}.
Our quasi-modular form theory in this example is attached to mirror quintic Calabi-Yau varieties and not the corresponding period domain.
There are other functions $\tau_1$ attached to one-dimensional families of varieties and the corresponding period maps.
They may have their own quasi-modular forms algebra different from the one explained in this section.

\begin{acknowledgement}
It would like to thank the organisers of the Fields workshop for inviting me to participate and give a talk. 
In particular, I would like to thank Charles F. Doran for his support and stimulating conversation.
\end{acknowledgement}

\def\cprime{$'$} \def\cprime{$'$} \def\cprime{$'$}

% \tableofcontents

\end{document}